\documentclass[]{article}
\usepackage{graphicx}
\usepackage{caption}
\usepackage{algorithmic}
\usepackage{algorithm}
\usepackage{amsmath}
\usepackage{amssymb}
\usepackage{amsfonts}
\usepackage{amsthm}
\usepackage{fullpage}
\usepackage{url}
\usepackage{color}
\usepackage{authblk} 
\usepackage{float}
\usepackage{mathabx} 
\usepackage{subfig}

\def\N{{\mathbb{N}}}
\def\R{{\mathbb{R}}}
\def\dt{{\textrm{d}t}}

\def\dy{{\textrm{d}y}}
\def\d{{\textrm{d}}}

\newtheorem{theorem}{Theorem}
\newtheorem{proposition}{Proposition}
\newtheorem{corollary}{Corollary}
\newtheorem{definition}{Definition}
\newtheorem{problem}{Problem}

\title{Numerical Analysis for Iterative Filtering\\ with New Efficient Implementations Based on FFT}
\author{Antonio Cicone, Haomin Zhou}

\begin{document}

\maketitle

\begin{abstract}
Real life signals are in general non--stationary and non--linear. The development of methods able to extract their hidden features in a fast and reliable way is of high importance in many research fields.
In this work we tackle the problem of further analyzing the convergence of the Iterative Filtering method both in a continuous and a discrete setting in order to provide a comprehensive analysis of its behavior.

Based on these results we provide new ideas for efficient implementations of Iterative Filtering algorithm which are based on Fast Fourier Transform (FFT), and the reduction of the original iterative algorithm to a direct method.

\end{abstract}

\section{Introduction}\label{sec:Intro}

Most real life signals are non--stationary and non--linear. Standard techniques like Fourier or wavelet transform prove to be unable to capture properly their hidden features \cite{cicone2017dummies}. For this reason Huang et al. proposed in 1998 a new kind of algorithm, called Empirical Mode Decomposition (EMD) \cite{huang1998empirical}, which allows to unravel the hidden features of a non--stationary signal $s(x)$, $x\in\R$,  by iteratively decomposing it into a finite sequence of simple components, called Intrinsic Mode Functions (IMFs). Such IMFs fulfill two properties:
the number of extrema and the number of zero crossings must either equal or differ at most by one;
considering upper and lower envelopes connecting respectively all the local maxima and minima of the function, their mean has to be zero at any point.

The wide variety of applications of this technique, see for instance \cite{cicone2017spectral,cicone2016hyperspectral,cicone2017Geophysics,sfarra2019thermal} and references therein, testified also by the high number of citations\footnote{The original work by Huang et al. \cite{huang1998empirical} as received so far, by itself, more than 10000 unique citations, according to Scopus} of the original paper by Huang et al. \cite{huang1998empirical}, together with the difficulty in analyzing it mathematically has attracted many researchers over the last two decades. Many alternative methods have been proposed, see \cite{cicone2017multidimensional} and references therein. All of these newly proposed methods are based on optimization with the only exception of the Iterative Filtering (IF) method, proposed by Lin et al. in \cite{lin2009iterative}, which is based instead on iterations.

The mathematical analysis of IF has been tackled by several authors in the last few years \cite{cicone2016adaptive,huang2009convergence,wang2013convergence,wang2012iterative} even for 2D or higher dimensional signals \cite{cicone2017multidimensional}. However several problems regarding this technique are still unsolved. In particular it is not yet clear how the stopping criterion used to discontinue the calculations of the IF algorithm influences the decomposition. Furthermore all the aforementioned analyses focused on the convergence of IF when applied to the extraction of a single IMF from a given signal, the so called inner loop. Regarding the decomposition of all the IMFs contained in a signal, which is related to the outer loop convergence and potential finiteness of the decomposition itself, nothing has been said so far. In this work we further analyze the IF technique addressing these and other questions.

The rest of this work is organized as follows: in Section \ref{sec:ContinuousIF} we review the details and properties of the method in the continuous setting and we provide new results regarding its inner loop convergence in presence of a stopping criterion as well as the outer loop convergence and finiteness. In Section \ref{sec:Discrete} we address the convergence analysis in the discrete setting for both the inner and outer loop of the algorithm. Based on these results in Section \ref{sec:speedup} we propose new ideas to increase the efficiency of the Iterative Filtering algorithm.

\section{IF algorithm in the continuous setting}\label{sec:ContinuousIF}

The key idea behind this decomposition technique is separating simple oscillatory components contained in a signal $s(x)$, $x\in\R$, the so called IMFs, by approximating the moving average of $s$ and subtracting it from $s$ itself. The approximated moving average is computed by convolution of $s$ with a window/filter function $w$

\begin{definition}\label{def:window}
    A filter/window $w$ is a nonnegative and even function in $C^0\left([-L,\ L]\right)$, $L>0$, and such that $\int_\R w(z)\d z=\int_{-L}^{L} w(z)\d z=1$.
\end{definition}

We point out that the idea of iteratively subtracting the moving average comes from the Empirical Mode Decomposition (EMD) method \cite{huang1998empirical} where the moving average was computed as a local average between an envelope connecting the maxima and one connecting the minima of the signal under study. The use of envelopes in an iterative way is the reason why the EMD algorithm is still lacking a rigorous mathematical framework.

The pseudocode of IF is given in Algorithm \ref{algo:IF}
\begin{algorithm}
\caption{\textbf{Iterative Filtering} IMF = IF$(s)$}\label{algo:IF}
\begin{algorithmic}
\STATE IMF = $\left\{\right\}$
\WHILE{the number of extrema of $s$ $\geq 2$}
      \STATE $s_1 = s$
      \WHILE{the stopping criterion is not satisfied}
                  \STATE  compute the filter length $l_m$ for $s_{m}(x)$
                  \STATE  $s_{m+1}(x) = s_{m}(x) -\int_{-l_m}^{l_m} s_m(x+t)w_m(t)\dt$
                  \STATE  $m = m+1$
      \ENDWHILE
      \STATE IMF = IMF$\,\cup\,  \{ s_{m}\}$
      \STATE $s=s-s_{m}$
\ENDWHILE
\STATE IMF = IMF$\,\cup\,  \{ s\}$
\end{algorithmic}
\end{algorithm}
where $w_m(t)$ is a nonnegative and compactly supported window/filter with area equal to one and support in $[-l_m,\ l_m]$, where $l_m$ is called filter length and represents the half support length.

The IF algorithm contains two loops: the inner and the outer loop, the second and first while loop in the pseudocode respectively. The former captures a single IMF, while the latter produces all the IMFs embedded in a signal.

Assuming $s_1=s$, the key step of the algorithm consists in computing the moving average of $s_m$ as
\begin{equation}\label{eq:Mov_Average}
\mathcal{L}_m(s_m)(x)=\int_{-l_m}^{l_m} s_m(x+t)w_m(t)\dt,
\end{equation}
which represents the convolution of the signal itself with the window/filter $w_m(t)$.

The moving average is then subtracted from $s_m$ to capture the fluctuation part as
\begin{equation}\label{eq:flactuations}
\mathcal{M}_{m}(s_m)= s_m-\mathcal{L}_m(s_m)=s_{m+1}
\end{equation}

The first IMF, $\textrm{IMF}_1$, is computed repeating iteratively this procedure on the signal $s_m$, $m\in\N$, until a stopping criterion is satisfied, as described in the following section.

To produce the $2$-nd IMF we apply the same procedure to the remainder signal $r=s-\textrm{IMF}_1$.
Subsequent IMFs are produced iterating the previous steps.

The algorithm stops when $r$ becomes a trend signal, meaning it has at most one local extremum.

We observe that, even thought the algorithm allows potentially to recompute the filter length $l_m$ at every step of each inner loop, in practice we always compute the filter length only at the first step of an inner loop and then we keep it constant throughout the subsequent iterations. Hence $l_m=l_1=l$ for every $m\geq 1$.

Following \cite{lin2009iterative}, one possible way of computing the filter length $l$ is given by the formula
\begin{equation}\label{eq:Unif_Mask_length}
l:=2\left\lfloor\nu \frac{N}{k}\right\rfloor
\end{equation}
where $N$ is the total number of sample points of a signal $s(x)$, $k$ is the number of its extreme points, $\nu$ is a tuning parameter usually fixed around 1.6, and $\left\lfloor \cdot \right\rfloor$ rounds a positive number to the nearest integer closer to zero. In doing so we are computing some sort of average highest frequency contained in $s$.

Another possible way could be the calculation of the Fourier spectrum of $s$ and the identification of its highest frequency peak. The filter length $l$ can be chosen to be proportional to the reciprocal of this value.

The computation of the filter length $l$ is an important step of the IF technique. Clearly, $l$ is strictly positive and, more importantly, it is based solely on the signal itself. This last property makes the method nonlinear.

In fact, if we consider two signals $p$ and $q$ where $p\neq q$, assuming $\textrm{IMFs}(\bullet)$ represent the decomposition of a signal into IMFs by IF, the fact that we choose the half support length based on the signal itself implies that in general

$$\textrm{IMFs}(p+q)\neq \textrm{IMFs}(p)+\textrm{IMFs}(q)$$

Regarding the convergence analysis of the Iterative Filtering inner loop we recall here the following theorem

\begin{theorem}[Convergence of the Iterative Filtering method \cite{cicone2016adaptive,huang2009convergence}]\label{thm:theo_1}
Given the filter function $w(t), t\in[-l,l]$ be $L^2$, symmetric, nonnegative, $\int_{-l}^l w(t)\dt = 1$ and let $s(x)\in L^2(\mathbb{R})$. \newline
If $|1- \widehat{w}(\xi)| < 1 $ or $\widehat{w}(\xi)=0$, where $\widehat{w}(\xi)$ is the Fourier transform of $w$ computed at the frequency $\xi$,

Then
$\{\mathcal{M}^m(s)\}$ converges and
\begin{equation}\label{eq:IMF_cont}
\textrm{IMF}_1 = \lim\limits_{m\rightarrow \infty}{\mathcal{M}^m(s)(x)}= \int_{-\infty}^{\infty} \widehat{s}(\xi) \chi_{\{\widehat{w}(\xi)=0 \}}
 e^{2\pi i \xi x} \textrm{d}\xi
\end{equation}
\end{theorem}

We observe here that given $h:[-\frac{l}{4},\frac{l}{4}]\rightarrow\R$, $z\mapsto h(z)$, nonnegative, symmetric, with $\int_\R h(z)\d z=\int_{-\frac{l}{4}}^{\frac{l}{4}} h(z)\d z=1$, if we construct the window $w_1$ as the convolution of $h$ with itself and we fix $w_m=w_1$ throughout all the steps $m$ of an inner loop, then the method converges for sure to the limit function \eqref{eq:IMF_cont} which depends only on the shape of the filter function chosen and the support length selected by the method \cite{cicone2016adaptive,cicone2017spectral}.

In general we can assume that the filter functions $w_m(u)$ are defined as some scaling of an a priori fixed filter shape $w:[-1,1]\rightarrow\R$. In particular we define the scaling function
\begin{equation}\label{eq:scalingFunc}
 g_m:[-1,1]\rightarrow [-l_m,l_m],\qquad  t\mapsto u=g_m(t),
\end{equation}
where $g_m$ is assumed to be invertible and monotone, such that $w_m(u)=C_m w(g_m^{-1}(u))=C_m w(t)$, where $t=g_m^{-1}(u)$, $u=g_m(t)$ and $C_m$ is a scaling coefficient which is required to ensure that $\int_\R w_m(u)\d u=\int_{-l_m}^{l_m} w_m(u)\d u=1$.

Regarding the computation of the scaling coefficient $C_m$, from the observation that $\d u = g_m'(t) \d t$, it follows that

\begin{equation}\label{eq:C_m}
\int_{-l_m}^{l_m} w_m(u)\d u=\int_{-l_m}^{l_m} C_m w(g_m^{-1}(u))\d u=C_m \int_{-1}^{1} w(t) |g_m'(t)| \d t
\end{equation}

hence

\begin{equation}\label{eq:C_x2}
C_m= \frac{1}{\int_{-1}^{1} w(t) |g_m'(t)| \d t}
\end{equation}

and

\begin{equation}\label{eq:w_m}
w_m(u)=C_m w(g_m^{-1}(u))=\frac{w(g_m^{-1}(u))}{\int_{-1}^{1} w(t) |g_m'(t)| \d t}
\end{equation}

As an example of a scaling function we can consider, for instance, linear or quadratic scalings: $g_m(t)=l_m t$ and $g_m(t)=l_m t^2$ respectively.

In the case of linear scaling we have that $g_m^{-1}(u)=\frac{u}{l_m}$, $g'_m(t)=l_m\geq 0$, for every $t\in\R$, and $C_m= \frac{1}{l_m}$. Hence

\begin{equation}\label{eq:w_m_linear}
w_m(u)=\frac{w\left(\frac{u}{l_m}\right)}{l_m}
\end{equation}

\subsection{IF inner loop convergence in presence of a stopping criterion}\label{sec:Stopping}

In Algorithm \ref{algo:IF} the inner loop has to be iterated infinitely many times. In numerical computations, however, some stopping criterion has to be introduced. One possible stopping criterion follows from the solution of
\begin{problem}\label{pb:IF_num_conv}
For a given $\delta > 0$ we want to find the value $N_0\in\N$ such that
 $$\|\mathcal{M}^N(s)(x)-\mathcal{M}^{N+1}(s)(x)\|_{L^2}<\delta \qquad \forall N\geq N_0$$
\end{problem}

Applying the aforementioned stopping criterion, the inner loop of  Algorithm \ref{algo:IF} converges in finite steps to an IMF whose explicit form is given in the following theorem where $\widehat{s}(\xi)$ represents the Fourier transform of $s$ at frequency $\xi$.

\begin{theorem}\label{thm:IF_inner_conv_stopping}
Given $s\in L^2(\R)$ and $w$ obtained as the convolution $\widetilde{w}\ast \widetilde{w}$, where $\widetilde{w}$ is a filter/window, Definition \ref{def:window}, and fixed $\delta>0$.

Then, for the minimum $N_0\in\N$ such that the following inequality holds true

\begin{equation}\label{eq:N0}
    \frac{N_0^{N_0}}{(N_0+1)^{N_0+1}}<\frac{\delta}{\left\|\widehat{s}(\xi)\right\|_{L^2}} \quad \forall \xi\in\R
\end{equation}
we have that $\left\| \mathcal{M}^N(s)(x)-\mathcal{M}^{N+1}(s)(x)\right\|_{L^2}<\delta \quad \forall N\geq N_0$ and the first IMF is given by

\begin{equation}\label{eq:IMF_IF_stop}
\textrm{IMF}_1^\textrm{SC}=\mathcal{M}^N(s)(x)=\int_{\R} (1-\widehat{w}(\xi))^N \widehat{s}(\xi) e^{2\pi i \xi x}\d \xi \quad \forall N\geq N_0
\end{equation}
\end{theorem}

\begin{proof}
From the hypotheses on the filter $w$ it follows that its Fourier transform is in the interval $[0,\ 1]$, see \cite{cicone2016adaptive}. Furthermore from the linearity of the Fourier transform it follows that

$${\widehat{\mathcal{M}^N(s)(x)}(\xi)}= (1- \widehat{w}(\xi))^N \widehat{s}(\xi)=\left\{
                                                                              \begin{array}{cc}                                                                                 \widehat{s}(\xi) &  \textrm{ if } \widehat{w}(\xi)=0\\
                                                                               (1- \widehat{w}(\xi))^N \widehat{s}(\xi)  & \textrm{ if } |1- \widehat{w}(\xi)|  < 1\\
                                                                              \end{array}
                                                                            \right.$$

since the Fourier Transform is a unitary operator, by the Parseval's Theorem, it follows that
$$\left\| \mathcal{M}^N(s)(x)-\mathcal{M}^{N+1}(s)(x)\right\|_{L^2} =\left\| \widehat{\mathcal{M}^N(s)(x)}(\xi)-\widehat{\mathcal{M}^{N+1}(s)(x)}(\xi)\right\|_{L^2}$$
$$=
\left\|(1-\widehat{w}(\xi))^N \left[1 -(1-\widehat{w}(\xi))\right]\widehat{s}(\xi)\right\|_{L^2}=
\left\|(1-\widehat{w}(\xi))^N \widehat{w}(\xi)\widehat{s}(\xi)\right\|_{L^2}
$$

We point out that this formula can also be interpreted as the $L^2$--norm of the moving average of $\mathcal{M}^N$ which is given by the convolution $\mathcal{M}^N\ast w$.

For a fixed $N$ we can compute the maximum of the function $(1-\widehat{w}(\xi))^N \widehat{w}$, for $\widehat{w}\in[0,\ 1]$, that is attained for $\widehat{w}(\xi)=\frac{1}{N+1}$. Therefore
$$\left\|(1-\widehat{w}(\xi))^N \widehat{w}(\xi)\widehat{s}(\xi)\right\|_{L^2}\leq \left\|\left(1-\frac{1}{N+1}\right)^N\frac{1}{N+1}\widehat{s}(\xi)\right\|_{L^2}$$
$$=\left\|\frac{N^N}{(N+1)^{N+1}}\widehat{s}(\xi)\right\|_{L^2}<\delta$$
Hence we consider the smallest $N_0\in\N$ such that
$$\frac{N_0^{N_0}}{(N_0+1)^{N_0+1}}<\frac{\delta}{\left\|\widehat{s}(\xi)\right\|_{L^2}}$$
\end{proof}

Equation \eqref{eq:IMF_IF_stop} provides a valuable insight on how the implemented algorithm is actually decomposing a signal into IMFs.
We recall that without any stopping criterion each IMF of a signal $s$ is given by the inverse Fourier transform of $\widehat{s}$ computed at the frequencies corresponding to zeros of $\widehat{w}$, as stated in \eqref{eq:IMF_cont}.

Therefore, from the observation that $\widehat{w}$ is a function not compactly supported and with isolated zeros, the IMFs produced with IF are given by the summation of pure and well separated tones.

Whereas, when we enforce a stopping criterion, we end up producing IMFs containing a much richer spectrum. In fact from \eqref{eq:IMF_IF_stop} we discover that an IMF is now given by the inverse Fourier transform of $\widehat{s}$ computed at every possible frequency in $\R$, each multiplied by the coefficient $(1-\widehat{w}(\xi))^N$. Since, by construction, $0 \leq \widehat{w}(\xi)\leq 1$, $\forall \xi\in\R$, then $(1-\widehat{w}(\xi))^N$ is equal to 1 if and only if  $\widehat{w}(\xi)=0$, whereas for all the other frequencies it is smaller than 1 and it tends to zero as $N$ grows. The $(1-\widehat{w}(\xi))^N$ quantity represents in practice the percentage with which each frequency is contained in the reconstruction of an IMF from the Fourier transform of the original signal. The higher is the number of iterations $N$ the narrower are the intervals of frequencies that are almost completely captured in each IMF. And as $N\rightarrow\infty$ such intervals coalesce into isolated points corresponding to the zeros of $\widehat{w}$.

\subsubsection{Convergence with a threshold}

We start recalling a few properties regarding the filter functions $w$. Assuming $w(x)$, $x\in\R$, is a filter function supported on $(-1,\ 1)$, if we use the linear scaling described in \eqref{eq:w_m_linear}, then we can construct
\begin{equation}\label{eq:w^a}
w^a(x)=\frac{1}{a} w\left(\frac{x}{a}\right)
\end{equation}
where $w^a(x)$ is supported on $(-a,\ a)$.

If we define $\widehat{w}(\xi)=\int_{-\infty}^{+\infty} w(x) e^{-i\xi x 2 \pi} \d x$, then
\begin{equation}\label{eq:fft_w^a}
\widehat{w^a}(\xi)=\int_{-\infty}^{+\infty} \frac{1}{a}w\left(\frac{x}{a}\right) e^{-i\xi \frac{x}{a} a 2 \pi} \d x=\widehat{w}(a\xi)
\end{equation}
Therefore, if $\xi_0$ is a root of $\widehat{w}(\xi)=0$, then $\frac{\xi_0}{a}$ is a root of $\widehat{w^a}(\xi)=0$ because $\widehat{w^a}\left(\frac{\xi_0}{a}\right)=\widehat{w}\left(a\frac{\xi_0}{a}\right)=\widehat{w}(\xi_0)=0$.

We remind that, since $w$ are compactly supported functions, their Fourier transform are defined on $\R$ and they have zeros which are isolated points.

Given $0 < \gamma < 1$, we identify the set
\begin{equation}\label{eq:I_gamma}
I_{w,\gamma, N}=\left\{\xi\in\R\  :\ \widehat{w}(\xi) \leq 1 - \sqrt[N]{1-\gamma}\right\}.
\end{equation}
As $N\rightarrow\infty$ the quantity $1 - \sqrt[N]{1-\gamma}\rightarrow 0$, therefore $I_{w,\gamma, N}$ coalesces into isolated points corresponding to the zeros of $\widehat{w}$.

If we consider filters like the Fokker-Planck filters \cite{cicone2016adaptive} or any filter with smooth finite support properties we must have that, for a fixed $N\in\N$ and $\gamma>0$, there exists $\Xi_0>0$ such that

\begin{equation}\label{eq:Xi_0}
\widehat{w}(\xi) \leq 1 - \sqrt[N]{1-\gamma} < 1 \qquad \textrm{ for all }\ |\xi|\geq \Xi_0
\end{equation}

In fact, since $\int|w(x)|^2\d x < + \infty$ with $w(x)$ smooth function, then $\int|\widehat{w}(\xi)|^2\d \xi < + \infty$ which implies that $\widehat{w}(\xi)$ decays as $|\xi|\rightarrow \infty$.

So for a filter $w$ with smooth finite support properties the set $I_{w,\gamma, N}$ is made up of a finite number of disjoint compact intervals, containing zeros of $\widehat{w}$, together with the intervals $(-\infty,\ -\Xi_0]$ and $[\Xi_0,\ \infty)$.

Furthermore if we scale these filters using a linear scaling with coefficient $a>1$ it follows from the previous observations that $\Xi_0\rightarrow 0$ and, as a consequence, $I_{w,\gamma, N}$ converges to $\R\backslash \{0\}$.

As an example of a compactly supported filter we can consider the triangular filter function
\begin{equation}\label{eq:triangular}
    w(x) = \left\{
             \begin{array}{cc}
               \frac{1}{L}-\frac{1}{L^2}|x| & \textrm{for} \; |x|\leq L\\
               0 & \textrm{otherwise} \\
             \end{array}
           \right.
\end{equation}
whose Fourier transform is
\begin{equation}\label{eq:triangular_fourier}
    \widehat{w}(\xi) = \frac{1}{L}\frac{\sin^2\left(L\pi \xi\right)}{\left(\pi \xi\right)^2}.
\end{equation}
The triangular filter and its Fourier transform are depicted in Fig. \ref{fig:Triangle_filter}

\begin{figure}%
    \centering
    \subfloat{{\includegraphics[width=0.48\textwidth]{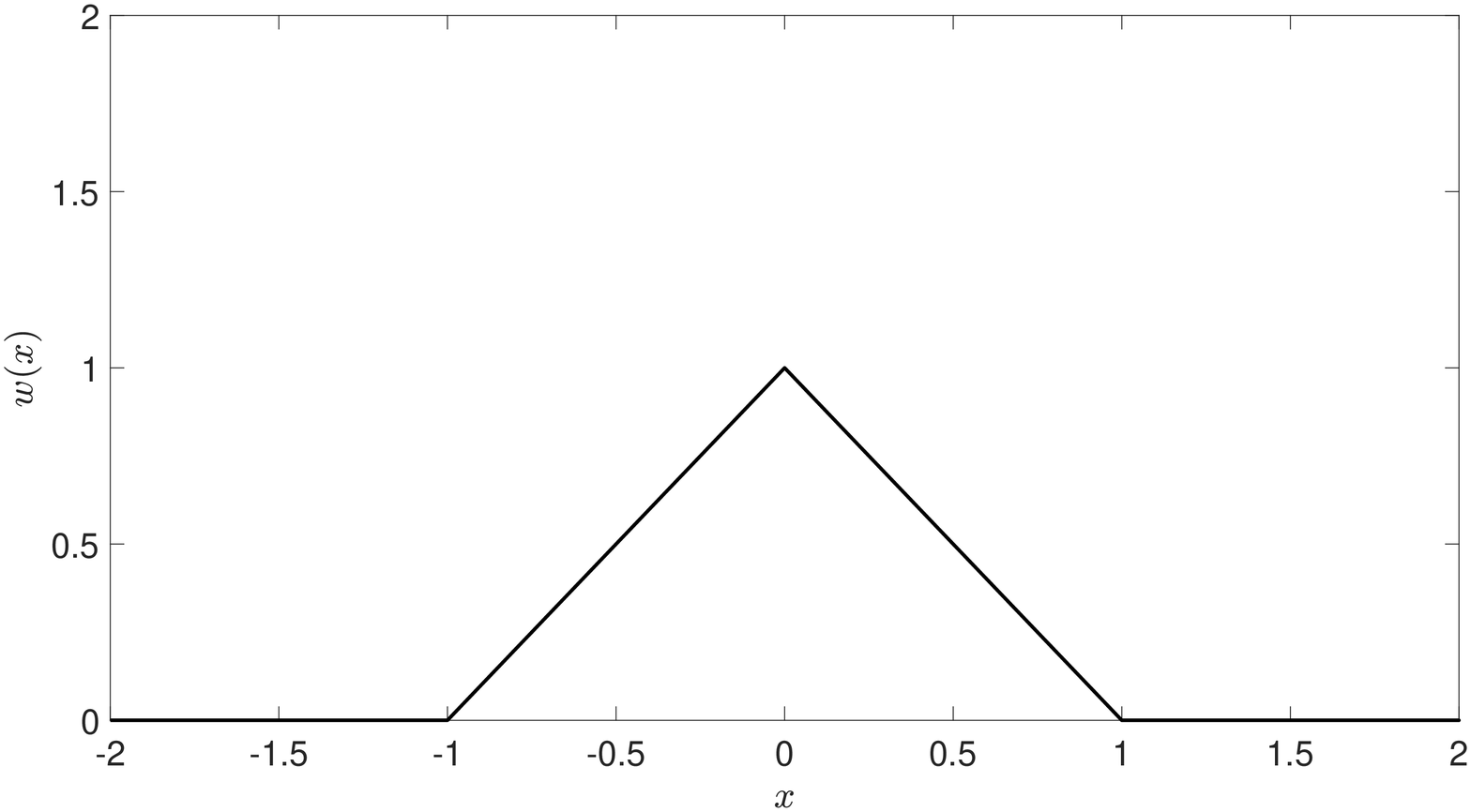} }}%
    ~
    \subfloat{{\includegraphics[width=0.48\textwidth]{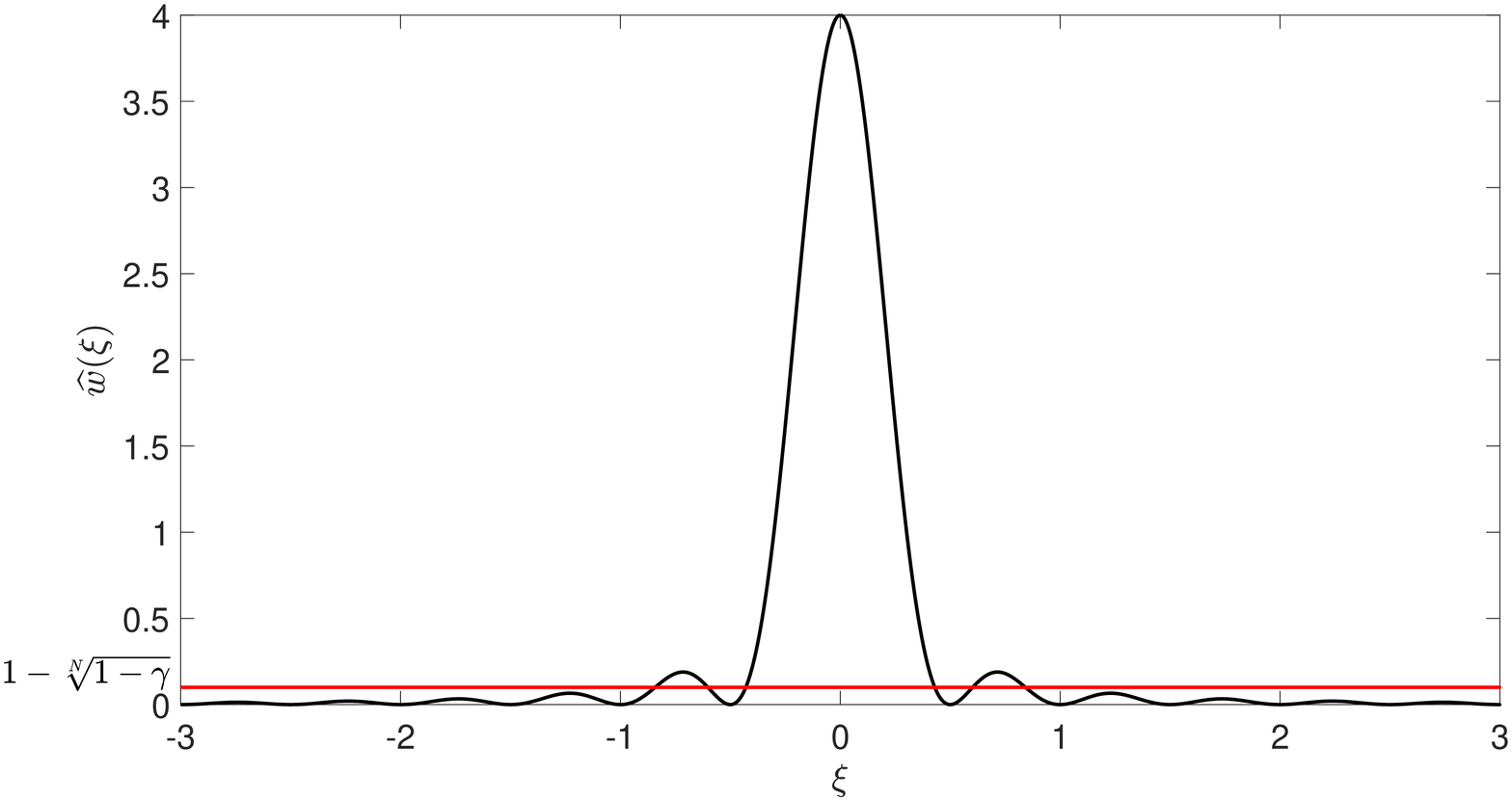} }}%
    \caption{ Left panel, triangular filter \eqref{eq:triangular} with $L=1$. Right panel, in black the Fourier transform \eqref{eq:triangular_fourier} and in red the threshold value $1 - \sqrt[N]{1-\gamma}$}\label{fig:Triangle_filter}
\end{figure}

Given the threshold value  $1 - \sqrt[N]{1-\gamma}$ depicted in the right panel of Fig. \ref{fig:Triangle_filter} and the triangular filter \eqref{eq:triangular} with $L=1$, the set $I_{w,\gamma, N}$ is made up of four intervals: two compactly supported and centered around $1/2$ and $-1/2$, and other two starting around $0.8$ and $-0.8$ and ending at infinity and minus infinity, respectively.

We can use the threshold value $1 - \sqrt[N]{1-\gamma}$ in the computation of an IMF as follows: given \eqref{eq:IMF_IF_stop}, whenever $(1-\widehat{w}(\xi))^N \geq 1-\gamma$, we substitute $\widehat{w}(\xi)$ with zero. This is equivalent to setting $\widehat{w}(\xi)=0$ whenever $\xi\in I_{w,\gamma, N}$.

Therefore, using the previously described thresholding and based on Theorem \ref{thm:IF_inner_conv_stopping}, Algorithm \ref{algo:IF} converges to
\begin{equation}\label{eq:IMF_cont_stop_threshold}
    \textrm{IMF}_1^\textrm{TH} =
    \int_{\R\backslash I_{w,\gamma, N}} (1-\widehat{w}(\xi))^N \widehat{s}(\xi) e^{2\pi i \xi x}\d \xi + \int_{I_{w,\gamma, N}} \widehat{s}(\xi) e^{2\pi i \xi x}\d \xi \quad \forall N\geq N_0
\end{equation}
where $I_{w,\gamma, N}$ is defined in \eqref{eq:I_gamma}.

We are now ready to prove the following

\begin{proposition}
Assuming that all the hypotheses of Theorem \ref{thm:IF_inner_conv_stopping} are fulfilled, then for every $\epsilon>0$ there exist a stopping criterion value $\delta>0$ and a threshold $0<\gamma<1$ such that
\begin{equation}\label{eq:NormInequalities_TH}
    \left\|\textrm{IMF}_1-\textrm{IMF}_1^\textrm{TH}\right\|\leq \frac{\epsilon}{2}, \qquad  \left\|\textrm{IMF}_1^\textrm{TH}-\textrm{IMF}_1^\textrm{SC}\right\|\leq \frac{\epsilon}{2}
\end{equation}
and
\begin{equation}\label{eq:NormInequality_SC}
    \left\|\textrm{IMF}_1-\textrm{IMF}_1^\textrm{SC}\right\|\leq \epsilon
\end{equation}
where $\textrm{IMF}_1$, $\textrm{IMF}_1^\textrm{SC}$, and $\textrm{IMF}_1^\textrm{TH}$ are defined in \eqref{eq:IMF_cont}, \eqref{eq:IMF_IF_stop}, and \eqref{eq:IMF_cont_stop_threshold} respectively.
\end{proposition}
\begin{proof}
First of all we have that
$$\left\|\textrm{IMF}_1-\textrm{IMF}_1^\textrm{SC}\right\|   \leq  \left\|\textrm{IMF}_1-\textrm{IMF}_1^\textrm{TH}\right\| + \left\|\textrm{IMF}_1^\textrm{TH}-\textrm{IMF}_1^\textrm{SC}\right\|$$
where
\begin{eqnarray}
  \nonumber \left\|\textrm{IMF}_1-\textrm{IMF}_1^\textrm{TH}\right\| & \leq &  \left\|\int_\R \widehat{s}(\xi) \chi_{\left\{\xi\in\R\ |\ \widehat{w}(\xi)=0 \right\}} e^{2\pi i \xi x} \d \xi - \int_{\R\backslash I_{w,\gamma, N}} \!\!\!\!\!\!\!\!(1-\widehat{w}(\xi))^N\widehat{s}(\xi) e^{2\pi i \xi x}\d \xi - \int_{I_{w,\gamma, N}} \widehat{s}(\xi) e^{2\pi i \xi x}\d \xi \right\| \leq \\
  & & \left\|\int_{\R\backslash I_{w,\gamma, N}} \!\!\!\!\!\!\!\!(1-\widehat{w}(\xi))^N\widehat{s}(\xi) e^{2\pi i \xi x}\d \xi\right\| + \left\|\int_{I_{w,\gamma, N}\backslash \left\{\xi\in\R\ |\ \widehat{w}(\xi)=0 \right\}} \!\!\!\!\!\!\!\!\!\!\!\!\!\!\!\!\!\!\!\!\!\!\!\!\!\!\!\!\!\!\widehat{s}(\xi) e^{2\pi i \xi x}\d \xi \right\|
\end{eqnarray}
and
\begin{eqnarray}
  \nonumber \left\|\textrm{IMF}_1^\textrm{TH}-\textrm{IMF}_1^\textrm{SC}\right\| & \leq & \left\|\int_{\R\backslash I_{w,\gamma, N}} \!\!\!\!\!\!\!\!(1-\widehat{w}(\xi))^N\widehat{s}(\xi) e^{2\pi i \xi x}\d \xi + \int_{I_{w,\gamma, N}} \widehat{s}(\xi) e^{2\pi i \xi x}\d \xi - \int_\R (1-\widehat{w}(\xi))^N\widehat{s}(\xi) e^{2\pi i \xi x}\d \xi\right\| \leq \\
  & & \left\|\int_{I_{w,\gamma, N}} \left[1-(1-\widehat{w}(\xi))^N\right]\widehat{s}(\xi) e^{2\pi i \xi x}\d \xi \right\|
\end{eqnarray}

From \eqref{eq:I_gamma} and the fact that $\int_{I_{w,\gamma, N}} \widehat{s}(\xi) e^{2\pi i \xi x}\d \xi \rightarrow \int_{ \left\{\xi\in\R\ |\ \widehat{w}(\xi)=0 \right\}} \widehat{s}(\xi) e^{2\pi i \xi x}\d \xi$ as $\gamma\rightarrow 0$ or $N\rightarrow\infty$, it follows that there exist $N_1\in\N$ big enough and $0<\gamma_1<1$ small enough such that
$$\left\|\int_{I_{w,\gamma_1, N_1}\backslash \left\{\xi\in\R\ |\ \widehat{w}(\xi)=0 \right\}} \!\!\!\!\!\!\!\!\!\!\!\!\!\!\!\!\!\!\!\!\!\!\!\!\!\!\!\!\!\!\!\!\!\!\!\!\!\! \widehat{s}(\xi) e^{2\pi i \xi x}\d \xi \right\| \leq \frac{\epsilon}{4}$$

Furthermore there exist $0<\gamma_2<1$ small enough and a $N_2\in\N$ so that
$$\left\|\int_{I_{w,\gamma_2, N_2}} \!\!\!\!\!\!\!\!\!\!\!\!\!\!\!\!\left[1-(1-\widehat{w}(\xi))^{N_2}\right]\widehat{s}(\xi) e^{2\pi i \xi x}\d \xi \right\|\leq \frac{\epsilon}{2}$$
in fact as $\gamma_2\rightarrow 0$ the interval $I_{w,\gamma_2, N_2}$ tends to the set of frequencies corresponding to the zeros of $\widehat{w}(\xi)$.
Given $\gamma=\min\left\{\gamma_1,\ \gamma_2\right\}$, then there exists $N_3\in\N$ big enough such that $(1-\widehat{w}(\xi))^{N_3}$ is small enough in order to have
$$\left\|\int_{\R\backslash I_{w,\gamma, N}} \!\!\!\!\!\!\!\!\!\!\!\!\!\!\!\!(1-\widehat{w}(\xi))^{N_3}\widehat{s}(\xi) e^{2\pi i \xi x}\d \xi\right\|\leq \frac{\epsilon}{4}$$

If we consider $N_0=\max\left\{N_1,\ N_2,\ N_3\right\}$ there exists $\delta>0$ such that \eqref{eq:N0} holds true for every $N\geq N_0$.
\end{proof}

This proposition implies that $\textrm{IMF}_1^\textrm{TH}$ can be as close as we like to both $\textrm{IMF}_1^\textrm{SC}$ and $\textrm{IMF}_1$ if we choose wisely the stopping criterion value $\delta$ and the threshold $\gamma$.

\subsection{IF outer loop convergence}\label{subsec:outerLoopIF}

We do have now all the tools needed to study the Iterative Filtering outer loop convergence.

\begin{definition}[Significant IMFs with respect to $\eta>0$]
Fixed $\eta>0$ and given a signal $s$ and its decomposition in IMFs obtained using Algorithm \ref{algo:IF}, then we define \emph{significant IMFs} with respect to $\eta$ all the IMFs whose $L^\infty$-norm is bigger than $\eta$.
\end{definition}

\begin{theorem}\label{thm:IF_outer_conv}
Given a signal $s\in L^\infty(\R)$, whose continuous frequency spectrum is compactly supported with upper limit $B>0$ and lower limit $b>0$, and such that $\|\widehat{s}\|_{\infty} = c<\infty$, chosen a filter $w$ produced as convolution of a filter with itself, fixed $\delta>0$ and $\eta>0$.

\noindent Then the inner loop of Algorithm \ref{algo:IF} converges to \eqref{eq:IMF_IF_stop} and the outer loop produces only a finite number $M\in\N$ of significant IMFs whose norm is bigger than $\eta$.
\end{theorem}

\begin{proof}
Let us consider the Fourier transform of the signal $s$. From the hypotheses it follows that $|\widehat s(\xi)| = 0$ for every $\xi \geq B$.

We can assume that Algorithm \ref{algo:IF} in the first step of its outer loop starts selecting a filter $w_1$ such that the zero of $\widehat{w_1}$ with smallest frequency is at $B$. We recall in fact that one of the possible way to choose the filter length is based on the Fourier transform of $s$, as explained in Section 2. Given $\delta>0$ we can identify $N_1\in\N$ such that \eqref{eq:N0} is fulfilled for every $N\geq N_1$.

Now, from the hypothesis that $\|\widehat{s}\|_{\infty} = c<\infty$ it follows there exists the upper bound $c$ on $\widehat s(\xi)$ uniformly on $\xi \in \R$. From the hypotheses on the filter function it follows that $0<\widehat{w_1}<1$, ref. end of Section 2 in \cite{cicone2016adaptive}. Furthermore, from the assumption on the lower bound $b$ and upper bound $B$ of the continuous frequency spectrum of $s$, the fact that $\left\| e^{2\pi i \xi x} \right\|_{\infty}\leq 1$ for every $x,\xi\in\R$, by definition of the interval $I_{w_1,\gamma, \widetilde{N}_1}$, and for every $\widetilde{N}_1\geq N_1$ and $0<\gamma<  \frac{\eta}{2c(B-b)}$, it follows that
\begin{eqnarray}
\nonumber \left\|\textrm{IMF}_1^\textrm{SC}-\textrm{IMF}_1^\textrm{TH}\right\|_{\infty} & \leq & \left\|\int_{[b,\ B]\cap I_{w_1,\gamma, \widetilde{N}_1}} \left[1-(1-\widehat{w_1}(\xi))^{\widetilde{N}_1}\right]\widehat{s}(\xi) e^{2\pi i \xi x}\d \xi \right\|_{\infty}  \leq \\
\nonumber & \leq &  \int_{[b,\ B] \cap I_{w_1,\gamma, \widetilde{N}_1}} \left\| \left[1-(1-\widehat{w_1}(\xi))^{\widetilde{N}_1}\right] \right\|_{\infty} \left\| \widehat{s}(\xi)  \right\|_{\infty} \left\| e^{2\pi i \xi x} \right\|_{\infty} \d \xi \leq  \\
 & \leq& c \int_{[b,\ B]\cap I_{w_1,\gamma, \widetilde{N}_1}} \left\| \left[1-(1-\widehat{w_1}(\xi))^{\widetilde{N}_1}\right] \right\|_{\infty} \d \xi \leq c\gamma (B-b) < \frac{\eta}{2}
\end{eqnarray}

In particular we point out that $I_{w_1,\gamma, \widetilde{N}_1}$, defined as in \eqref{eq:I_gamma}, covers the interval of frequencies $[B-r_1,\ B+r_1]$, for some $r_1>\varepsilon>0$.

This last inequality follows from  the fact that if we scale linearly the filter function $w$ to enlarge its support, as in \eqref{eq:w^a} for $a>1$, its Fourier transform is proportionally shrunk \eqref{eq:fft_w^a}. However the signal $s$ does have a lower bound $b$ in the continuous frequency spectrum which implies that the filter function $w_1$ cannot have a too wide support and as a consequence its Fourier transform cannot be too much squeezed. Therefore it does exist $\varepsilon>0$ which lower bounds the radius $r_1$.

If $\left\|\textrm{IMF}_1^\textrm{TH}\right\|_{\infty}< \frac{\eta}{2}$ then we can for sure regard this component as not significant because $\left\|\textrm{IMF}_1^\textrm{SC}\right\|_{\infty}\leq \left\|\textrm{IMF}_1^\textrm{SC}-\textrm{IMF}_1^\textrm{TH}\right\|_{\infty} + \left\|\textrm{IMF}_1^\textrm{TH}\right\|_{\infty}< \eta$. Otherwise, assuming $\left\|\textrm{IMF}_1^\textrm{TH}\right\|_{\infty} \geq \frac{\eta}{2}$, if $\left\|\textrm{IMF}_1^\textrm{SC}\right\|_{\infty}\geq \eta$, then $\textrm{IMF}_1^\textrm{SC}$ represents the first significant IMF in the decomposition. This conclude the first step of the outer loop in Algorithm \ref{algo:IF}.

In the second step of the outer loop Algorithm \ref{algo:IF} iterates the previous passages using now the remainder signal $s_2=s-\textrm{IMF}_1^{SC}$ and selecting a filter $w_2$ such that the zero of $\widehat{w_2}$ with smallest frequency is at $B-r_1$.

Also in this case, given $\delta>0$, we can identify $N_2\in\N$ such that \eqref{eq:N0} is fulfilled for every $N\geq N_2$. Furthermore  $\widehat{s}_2(\xi)=\widehat{s}(\xi)-\widehat{\textrm{IMF}}_1^{SC}(\xi)=\left[1-\left(1-\widehat{w_2}(\xi)\right)^{\widetilde{N}_1}\right]\widehat{s}(\xi),\ \forall \xi \in \R$ which implies that
\begin{equation}
\left\|\widehat{s}_2\right\|_{\infty}\leq \left\|\left[1-\left(1-\widehat{w_2}(\xi)\right)^{\widetilde{N}_1}\right]\right\|_{\infty}\left\|\widehat{s}(\xi)\right\|_{\infty}\leq \left\|\widehat{s}(\xi)\right\|_{\infty}
\end{equation}
since  $\widehat{w_2}(\xi)\in [0,\ 1],\ \forall \xi \in \R$ \cite{cicone2016adaptive}. Hence $\widehat{s}_2$ has the same uniform upper bound $c$ over all $\xi \in \R^+$ as $\widehat s(\xi)$.

Therefore
\begin{equation}\label{eq:IMF2}
\left\|\textrm{IMF}_2^\textrm{SC}-\textrm{IMF}_2^\textrm{TH}\right\|_{\infty}\leq \left\|\int_{I_{w_2,\gamma, \widetilde{N}_2}}\!\!\!\!\!\!\!\!\!\!\!\!\!\!\!\!\! \left[1-(1-\widehat{w_2}(\xi))^{\widetilde{N}_2}\right]\widehat{s_2}(\xi) e^{2\pi i \xi x}\d \xi \right\|_{\infty} \leq c\gamma (B-b) < \frac{\eta}{2}
\end{equation}
for every $\widetilde{N}_2\geq N_2$ and $0<\gamma<  \frac{\eta}{2c(B-b)}$.

Furthermore $I_{w_2,\gamma, \widetilde{N}_2}$ covers the interval of frequencies $[B-r_2,\ B+r_2]$, for some $r_2>\varepsilon>0$. This last inequality follows from the same reasoning as before and the fact that the lower bound on the continuous frequency spectrum of $s_2$ is again $b$, by construction of $s_2$, the fact that $\gamma$ is fixed for every IMF and the Fourier transform of the scaled filter $w_2$ is a squeezed version of $\widehat w$, ref. equation \eqref{eq:fft_w^a}.

If $\left\|\textrm{IMF}_2^\textrm{TH}\right\|_{\infty}< \frac{\eta}{2}$ then we can regard this component as not significant. If instead $\left\|\textrm{IMF}_2^\textrm{TH}\right\|_{\infty}\geq \frac{\eta}{2}$ and $\left\|\textrm{IMF}_2^\textrm{SC}\right\|_{\infty}\geq \eta$, then $\textrm{IMF}_2^\textrm{SC}$ represents another significant IMF in the decomposition.

The subsequent outer loop steps follow similarly. The existence of the lower limit $\varepsilon$ for all $r_k>0$, $k\geq 1$, ensures that we can have a finite coverage of the interval of frequencies $[b,\ B]$.
In particular the algorithm generates a set $\left\{r_k\right\}_{k=1}^{R}$ such that $\sum_{k=1}^R r_k = B-b$ and there exists a natural number $0\leq M\leq R$ which represents the number of significant IMFs with respect to $\eta$.
\end{proof}

We point out that this theorem holds true also if we consider the $L^2$-norm instead of the $L^\infty$-norm thanks to the inclusion of $L^p$ spaces on a finite measure space.

From this Theorem it follows that IF with a stopping criterion allows to decompose a signal into a finite number of components given by \eqref{eq:IMF_IF_stop} each of which contains frequencies of the original signal filtered in a smart way.

We observe also that this theorem, together with Theorems \ref{thm:theo_1} and \ref{thm:IF_inner_conv_stopping}, allow to conclude that the IF method can not produce fake oscillations. Each IMF is in fact containing part of the oscillatory content of the original signal, as described in \eqref{eq:IMF_cont} and \eqref{eq:IMF_IF_stop}.

\section{IF algorithm in the discrete setting}\label{sec:Discrete}

Real life signals are discrete and compactly supported, therefore we want to analyze the IF algorithm discretization and study its properties.

Consider a signal $s(x)$, $x\in\R$, we assume for simplicity it is supported on $[0,\ 1]$, sampled at $n$ points $x_j= \frac{j}{n-1}$, with $j= 0,\ldots, n-1$, with a sampling rate which allows to capture all its fine details, so that aliasing will not play any role. The goal is to decompose the vector $\left[s(x_j)\right]_{j=0}^{n-1}$ into vectorial IMFs. Without loosing generality we can assume that $\|\left[s(x_j)\right]\|_2=1$.

From now on, to simplify the formulas, we use the notation $s=\left[s(x_j)\right]_{j=0}^{n-1}$. Furthermore, if not specified differently, we consider as matrix norm the so called Frobenius norm $\|A\|_2=\sqrt{\sum_{i,\ j=0}^{n-1}\left|a_{ij}\right|^2}$ which is unitarily invariant.

\begin{definition}\label{def:Discrete_filter}
A vector $w\in\R^n$, $n$ odd number, is called a \textbf{filter} if its values are symmetric with respect to the middle, nonnegative, and $\sum_{p=1}^n w_p  = 1$.
\end{definition}

We assume that a filter shape has been selected a priori, like one of the Fokker-Planck filters described in \cite{cicone2016adaptive}, and that some invertible and monotone scaling function $g_m$ has been chosen so that $w_m(\xi)$ can be computed as described in \eqref{eq:w_m}.
Therefore, assuming $s_1=s$, the main step of the IF method becomes
\begin{equation}\label{eq:s_m+1}
s_{m+1}(x_i) = s_{m}(x_i)-\int_{x_i-l_m}^{x_i+l_m} \!\!\!\!\!\!\! s_m(y)w_m(x_i-y)\dy\approx s_{m}(x_i)-\!\!\!\!\! \sum_{x_j=x_i-l_m}^{x_i+l_m}\!\!\!\!\! s_m(x_j)w_m(x_i-x_j)\frac{1}{n}, \quad j=0,\ldots,n-1
\end{equation}

In matrix form we have

\begin{equation}\label{eq:MatrixForm}
    s_{m+1}=(I-W_m)s_m
\end{equation}
where
\begin{equation}\label{eq:K}
    W_m=\left[w_m(x_i-x_j)\cdot \frac{1}{n}\right]_{i,\ j=0}^{n-1}=\left[\frac{w(g_{m}^{-1}(x_i-x_j))}{\sum_{z_r=-1}^{1} w(z_r) |g'_{m}(z_r)| \Delta z_r}\cdot \frac{1}{n}\right]_{i,\ j=0}^{n-1}
\end{equation}

Algorithm \ref{algo:IF_discrete} provides the discrete version of Algorithm \ref{algo:IF}

\begin{algorithm}
\caption{\textbf{Discrete Iterative Filtering} IMF = DIF$(s)$}\label{algo:IF_discrete}
\begin{algorithmic}
\STATE IMF = $\left\{\right\}$
\WHILE{the number of extrema of $s$ $\geq 2$}
      \STATE $s_1 = s$
      \WHILE{the stopping criterion is not satisfied}
                  \STATE compute the function $w_m(\xi)$, whose half support length $l_m$ is based on the signal $\left[s_m(x_i)\right]_{i=0}^{n-1}$
                  \STATE  $s_{m+1}(x_i) = s_{m}(x_i) - \sum_{j=0}^{n-1} s_m(x_j)w_m(|x_i-x_j|) \frac{1}{n},\qquad i= 0,\ldots, n-1$
                  \STATE  $m = m+1$
      \ENDWHILE
      \STATE IMF = IMF$\,\cup\,  \{ s_{m}\}$
      \STATE $s=s-s_{m}$
\ENDWHILE
\STATE IMF = IMF$\,\cup\,  \{ s\}$
\end{algorithmic}
\end{algorithm}

We remind that the first while loop is called outer loop, whereas the second one inner loop.

The first IMF is given by $\textrm{IMF}_1=\lim_{m\rightarrow\infty} (I-W_m)s_m$, where we point out that the matrix $W_m=[w_m(x_i-x_j)]_{i,\ j=0}^{n-1}$ depends on the half support length $l_m$ at every step $m$.

However in the implemented code the value $l_m$ is usually computed only in the first iteration of each inner loop and then kept constant in the subsequent steps, so that the matrix $W_m$ is equal to $W$ for every $m\in\N$. So the first IMF is given by

\begin{equation}\label{eq:First_IMF_fixed_length}
\textrm{IMF}_1=\lim_{m\rightarrow\infty} (I-W)^{m} s
\end{equation}

Furthermore in the implemented algorithm we do not let $m$ to go to infinity, instead we use a stopping criterion as described in section \ref{sec:Stopping}. For instance, we can define the following quantity
\begin{equation}\label{eq:SD}
SD:=\frac{\|s_{m+1}-s_{m}\|_2}{\|s_{m}\|_2}
\end{equation}
and we can stop the process when the value $SD$ reaches a certain threshold. Another possible option is to introduce a limit on the maximal number of iterations for all the inner loops. It is always possible to adopt different stopping criteria for different inner loops.

If we consider the case of linear scaling, making use of \eqref{eq:w_m_linear}, the matrix $W_m$ becomes

\begin{equation}\label{eq:W_linear}
    W_m=\left[\frac{w\left(\frac{x_i-x_j}{l_m}\right)}{l_m}\cdot \frac{1}{n}\right]_{i,\ j=0}^{n-1}=\left[\frac{w\left(\frac{i-j}{(n-1)l_m}\right)}{l_m}\cdot \frac{1}{n}\right]_{i,\ j=0}^{n-1}
\end{equation}

We point out that the previous formula represent an ideal $W_m$, however we need to take into account the quadrature formula we use to compute the numerical convolution in order to build the appropriate $W_m$ to be used in the DIF algorithm.

For instance, if we use the rectangle rule, we need to substitute the exact value of $w(y)$ at $y$ with its average value in the interval of length $\frac{1}{n}$ centered in $y$ and multiply this value for the length of  interval itself. Furthermore we should handle appropriately the boundaries of the support of $w(y)$, in fact the half length of the support is, in general, a non integer value. This can be done by handling separately the first and last interval in the quadrature formula. In fact we can scale the value of the integral on these two intervals proportionally to the actual length of the intervals themselves.

If we take into account all the aforementioned details we can reproduce a matrix $W_m$ which is row stochastic.

We observe that in the implemented code we simply scale each row of $W_m$ by its sum so that the matrix becomes row stochastic.

\subsection{Spectrum of $W_m$}

Since $W_m\in\R^{n\times n}$ represents the discrete convolution operator, it can be a circulant  matrix, Toeplitz  matrix or it can have a more complex structure. Its structure depends on the way we extend the signal outside its boundaries.

From now on we assume for simplicity that $n$ is an odd natural number, and that we have periodical extension of signals outside the boundaries, therefore $W_m$ is a circulant matrix given by

\begin{equation}\label{eq:W_m}
   W_m=\left[
         \begin{array}{cccc}
           c_0 & c_{n-1} & \ldots & c_1 \\
           c_{1} & c_0 & \ldots & c_2 \\
           \vdots & \vdots & \ddots & \vdots \\
           c_{n-1} & c_{n-2} & \ldots & c_0 \\
         \end{array}
       \right]
\end{equation}
where $c_j\geq 0$, for every $j=0,\ldots, \ n-1$, and $\sum_{j=0}^{n-1}c_j=1$.
Each row contains a circular shift of the entries of a chosen vector filter $w_m$. For the non periodical extension case we refer the reader to \cite{cicone2017BC}.

Denoting by $\sigma(W_m)$ the spectrum of the matrix, in the case of a circulant matrix it is well known that the eigenvalues $\lambda_j\in\sigma(W_m)$, $j=0,\ldots, \ n-1$ are given by the formula

\begin{equation}\label{eq:eig}
   \lambda_j=c_0+c_{n-1}\omega_j+\ldots+c_1 \omega_j^{n-1},\quad \textrm{ for } \qquad j=0,\ldots, \ n-1
\end{equation}
where $i=\sqrt{-1}$, and $\omega_j=e^\frac{2\pi i j}{n}$ $j$--th power of the $n$--th root of unity, for $j=0,\ldots, \ n-1$.

Since we construct the matrices $W_m$ using symmetric filters $w_m$, we have that $c_{n-j}=c_j$ for every $j=1,\ldots,\frac{n-1}{2}$. Hence $W_m$ is circulant, symmetric and
\begin{equation*}
   \lambda_j=c_0+c_{1}\left(\omega_j+\omega_j^{n-1}\right)+c_{2}\left(\omega_j^2+\omega_j^{n-2}\right)\ldots+
   c_{\frac{n-1}{2}}\left(\omega_j^\frac{n-1}{2}+\omega_j^{\frac{n+1}{2}}\right)=
\end{equation*}
\begin{equation*}
   c_0+\sum_{k=1}^{\frac{n-1}{2}}c_{k}\left(\omega_j^k+\omega_j^{n-k}\right)=c_0+\sum_{k=1}^{\frac{n-1}{2}}c_{k}\left(e^{\frac{2\pi i j}{n}k}+e^{\frac{2\pi i j}{n}(n-k)}\right)=
\end{equation*}
\begin{equation}\label{eq:eig1}
   c_0+\sum_{k=1}^{\frac{n-1}{2}}c_{k}\left(e^{\frac{2\pi i j}{n}k}-e^{\frac{2\pi i j}{n}k}e^{2\pi i j}\right)
\end{equation}
Therefore
\begin{equation}\label{eq:eig2}
   \lambda_j = c_0+2\sum_{k=1}^{\frac{n-1}{2}}c_{k} \cos\left(\frac{2\pi j k}{n}\right),\quad \textrm{ for } \qquad j=0,\ldots, \ n-1
\end{equation}

It is evident that, for any $j=0,\ldots, \ n-1$, $\lambda_j$ is real and $\sigma(W_m)\subseteq [-1,\ 1]$ since $W_m$ is a stochastic matrix.

Furthermore, if we make the assumption that the filter half supports length is always $l_m\leq \frac{n-1}{2}$, then the entries $c_j$ of the matrix $W_m$ are going to be zero at least for any $j\in[\frac{n-1}{4},\frac{3}{4}(n-1)]$.

We observe that the previous assumption is reasonable since it implies that we can study oscillations with periods at most equal to half of the length of a signal.

\begin{theorem}\label{thm:Spectra}
Considering the circulant matrix $W_m$ given in \eqref{eq:W_m}, assuming that $n>1$, $\sum_{j=0}^{n-1}c_j=1$, $c_j\geq 0$, and $c_{n-j}=c_j$, for every $j=1,\ldots,n-1$.

Then $W_m$ is non--defective, diagonalizable and has real eigenvalues.

Furthermore, if the filter half supports length $l_m$ is small enough so that $c_0=1$ and $c_j=0$, for every $j=1,\ldots, \ n-1$,  then we have $n$ eigenvalues $\lambda_j$ all equal $1$.

Otherwise, if the filter half supports length $l_m$ is big enough so that $c_0<1$ and the values $c_k$ correspond to the discretization of a function with compact and connected support, then there is one and only one eigenvalue equal to $1$, which is $\lambda_0$, all the other eigenvalues $\lambda_j$ are real and strictly less than one in absolute value. So they belong to the interval $(-1,1)$.
\end{theorem}

\begin{proof}
First of all we recall that symmetric matrices are always non--defective, diagonalizable and with a real spectrum.

In the case of $c_0=1$ the conclusion follows immediately from the observation that $W_m$ reduces to an identity matrix.

When $c_0<1$ from \eqref{eq:eig2} it follows that $\lambda_0=1$ and all the other eigenvalues belong to the interval $[-1,1]$. Let us assume, by contradiction, that there exists another eigenvalue $\lambda_d=1$ for some $d\in\left\{1,\ 2,\ldots,\ n-1\right\}$. We assume for simplicity that $n$ is odd. The proof in the even case works in a similar way.

From \eqref{eq:eig2} and the fact that $c_{n-j}=c_j$, for every $j=1,\ldots,\frac{n-1}{2}$, it follows that
\begin{equation}\label{eq:lambda_d}
   \lambda_d = c_0+2\sum_{k=1}^{\frac{n-1}{2}}c_{k} \cos\left(\frac{2\pi d k}{n}\right),\quad \textrm{ for } \qquad  d\in\left\{1,\ 2,\ldots,\ n-1\right\}
\end{equation}
In the right hand side we have among the terms $c_{k}$, which by themselves would add up to 1, at least $c_1>0$ which is multiplied by $\cos\left(\frac{2\pi d}{n}\right)<1$ for any $d\in\left\{1,\ 2,\ldots,\ n-1\right\}$. Therefore the right hand side will never add up to 1. Hence we have a contradiction.

From \eqref{eq:eig2} it follows also that $\lambda_d\neq -1$ for any $d\in\left\{1,\ 2,\ldots,\ n-1\right\}$ because $\lambda_d$ is given by a convex combination of cosines and $+1$.

So all the eigenvalues of $W_m$ except $\lambda_0$ are real and strictly less than one in modulus.
\end{proof}

We observe that in the discrete iterative filtering algorithm the entries $c_k$ derive from the discretization of a filter function which is by Definition \ref{def:window} compactly supported. Furthermore, since the filter function is used to compute the moving average of a signal, it is reasonable to require its support to be connected.

Form this theorem it follows that

\begin{corollary}\label{cor:SpectraDoubleConvolution}
Considering the matrix $W_m$ given in the previous theorem, assuming $c_0<1$ and that $W_m$ is constructed using a filter $w_m$ that is produced as convolution of a symmetric filter $h_m$ with itself, then there is one and only one eigenvalue equal to $1$, all the other eigenvalues belong to the interval $[0,1)$.
\end{corollary}

\begin{proof}
The proof follows directly from the previous theorem and the fact that the matrix $W_m=\widetilde{W}_m^T*\widetilde{W}_m=\widetilde{W}_m^2$, where $\widetilde{W}_m$ is a circulant symmetric convolution matrix associated with the filter $\widetilde{w}_m$.
\end{proof}

\begin{corollary}\label{cor:Kernel_I-W}
Assuming $c_0<1$, the eigenvector of $W_m$ corresponding to $\lambda_0=1$ is a basis for the kernel of the matrix $(I-W_m)$, which has dimension one.
\end{corollary}

Before presenting the main proposition we recall that, given a circulant matrix $C=\left[c_{pq}\right]_{p,\ q=0,\ldots,n-1}$, its eigenvalues are
\begin{equation}\label{eq:Lambdas}
\lambda_p\ =\ \sum_{q=0}^{n-1} c_{1q}e^{-2\pi i p \frac{q}{n}}  \qquad  \qquad p\ =\ 0,\ldots,\ n-1
\end{equation}
and the corresponding eigenvectors are
\begin{equation}\label{eq:Eigenvectors}
u_p\ =\ \frac{1}{\sqrt{n}} \left[1,\ e^{-2\pi i p\frac{1}{n}},\ldots,\ e^{-2\pi i p\frac{n-1}{n}}\right]^T \qquad  \qquad p\ =\ 0,\ldots,\ n-1
\end{equation}
which form an orthonormal set.

We recall that an eigenvalue of a matrix is called semisimple whenever its algebraic multiplicity coincides with its geometric multiplicity.

\begin{proposition}\label{pro:LimitValue}

Given a matrix $W_m$, assuming  that all the assumptions of Theorem \ref{thm:Spectra} and Corollary \ref{cor:SpectraDoubleConvolution} hold true,
and assuming that $W_m=W$ for any $m\geq 1$. Given $\left\{\lambda_p\right\}_{p=0,\ldots,n-1}$, semisimple eigenvalues of $W$, and the corresponding eigenvectors $\left\{u_p\right\}_{p=0,\ldots,n-1}$, we define the matrix $U$ having as columns the eigenvectors $u_p$. Assuming that $W$ has $k$ zero eigenvalues, where $k$ is a number in the set $\in\{0,\ 1,\ldots,\ n-1\}$,

Then
\begin{equation}\label{eq:discreteLimit}
\lim_{m\rightarrow \infty}(I-W)^m=U Z U^T
\end{equation}
where $U$ is unitary and $Z$ is a diagonal matrix with entries all zero except $k$ elements in the diagonal which are equal to one.
\end{proposition}

\begin{proof}
From Theorem \ref{thm:Spectra} we know that $W$ is diagonalizable, therefore the matrix $U$ is orthogonal and all the eigenvalues of $W$ are semisimple. Furthermore, since the eigenvectors of $W$ are orthonormal, it follows that $U$ is a unitary matrix. Hence $W=U D U^T$, where $D$ is a diagonal matrix containing in its diagonal the eigenvalues of $W$.
From the assumption that $W$ is associated with a double convolved filter it follows that the spectrum of $W$ is contained in $[0,1]$, ref. Corollary \ref{cor:SpectraDoubleConvolution}. Therefore also the spectrum of $(I-W)$ is contained in $[0,1]$. Furthermore
\begin{equation*}
    (I-W)=U(I-D)U^T
\end{equation*}
and $I-D$ is a diagonal matrix whose diagonal entries are in the interval $(0,1)$ except the first one which equals 0, ref. Corollary \ref{cor:Kernel_I-W}, and $k$ entries that are equal to 1.
Hence
\begin{equation*}
\lim_{m\rightarrow \infty}(I-W)^m=\lim_{m\rightarrow \infty}U(I-D)^m U^T= U Z U^T
\end{equation*}
where $Z$ is a diagonal matrix with entries all zero except $k$ elements in the diagonal which are equal to one.
\end{proof}

From the previous proposition it follows
\begin{corollary}\label{cor:ExplicitFormulaDiscreteIMF}
Given a signal $s\in\R^n$, assuming that we are considering a doubly convolved filter, and the half filter support length is constant throughout all the steps of an inner loop,

Then the first outer loop step of the DIF method converges to
\begin{equation}\label{eq:discreteIMF}
\textrm{IMF}_1=\lim_{m\rightarrow \infty}(I-W)^m s=U Z U^Ts
\end{equation}
\end{corollary}

So the DIF method in the limit produces IMFs that are projections of the given signal $s$ onto the eigenspace of $W$ corresponding to the zero eigenvalue which has algebraic and geometric multiplicity $k\in\{0,\ 1,\ldots,\ n-1\}$. Clearly, if $W$ has only a trivial kernel then the method converges to the zero vector. We point out that since (\ref{eq:Lambdas}) is also the Discrete Fourier Transform (DFT) formula of the sequence $\{c_{1q}\}_{q=0,\ldots,n-1}$, where $C=[c_{pq}]$ is a circulant matrix, it follows that the eigenvalues of $W$,  can be computed directly as the DFT of the sequence $\{w_{1q}\}_{q=0,\ldots,n-1}$, by means of the Fast Fourier Transform (FFT). If we regard the DFT as a discretization of the Fourier Transform of the filter function $w$ it becomes clear that, since the latter has only isolated zeros, in many cases we will not have eigenvalues exactly equal to zero. So in general $W$ has only a trivial kernel and \eqref{eq:discreteIMF} converges to the zero vector. In order to ensure that the method produces a non zero vector we need to discontinue the calculation introducing some stopping criterion.

\subsection{DIF inner and outer loop convergence in presence of a stopping criterion}

If we assume that the half support length $l_m$ is computed only in the beginning of each inner loop, then the first IMF is given by \eqref{eq:First_IMF_fixed_length} and \eqref{eq:discreteIMF}.

In order to have a finite time method we may introduce a stopping criterion in the DIF algorithm, like the condition
\begin{equation}\label{eq:Discrete_Abs_StopCond}
\|s_{m+1}-s_m\|_2<\delta \qquad \forall m\geq N_0
\end{equation}
for some fixed $\delta>0$

Then, based on Corollary \ref{cor:ExplicitFormulaDiscreteIMF}, we produce an approximated first IMF given by

\begin{equation}\label{eq:approx_first_IMF}
    \overline{\textrm{IMF}}_1=(I-W)^{N_0} s = U(I-D)^{N_0} U^T s
\end{equation}

\begin{theorem}\label{thm:DIF_conv_stop}
    Given $s\in\R^n$, we consider the convolution matrix $W$ defined in \eqref{eq:W_m}, associated with a filter vector $w$ given as a symmetric filter $h$ convolved with itself. Assuming that $W$ has $k$ zero eigenvalues, where $k$ is a number in the set $\in\{0,\ 1,\ldots,\ n-1\}$, and fixed $\delta>0$,

Then, calling $\widetilde{s}=U^T s$, for the minimum $N_0\in\N$ such that it holds true the inequality

\begin{equation}\label{eq:N0_discrete}
    \frac{N_0^{N_0}}{\left(N_0+1\right)^{N_0+1}}<\frac{\delta}{\|\widetilde{s}\|_\infty{\sqrt{n-1-k}}}
\end{equation}
we have that $\left\| s_{m+1}-s_m\right\|_{2}<\delta \quad \forall m\geq N_0$ and the first IMF is given by
\begin{equation}\label{eq:IMF1_direct}
\overline{\textrm{IMF}}_1=U(I-D)^{N_0} U^T s= U P \left[
                                          \begin{array}{ccccccc}
                                             0 &   &   &   &   &   &   \\
                                              &  (1-\lambda_1 )^{N_0} &   &   &   &   &   \\
                                              &   & \ddots  &   &   &   &   \\
                                              &   &   &  (1-\lambda_{n-1-k} )^{N_0} &   &   &   \\
                                              &   &   &   &  1 &   &   \\
                                              &   &   &   &   &  \ddots &   \\
                                              &   &   &   &   &   & 1  \\
                                          \end{array}
                                        \right] P^T U^T s
\end{equation}
where $P$ is a permutation matrix which allows to reorder the columns of $U$, which correspond to eigenvectors of $W$, so that the corresponding eigenvalues $\{\lambda_p\}_{p=1,\ldots,\ n-1}$ are in decreasing order.
\end{theorem}

\begin{proof}
\begin{eqnarray}
 \nonumber \|s_{m+1}-s_m\|_2 &=& \|(I-W)^{m+1}-(I-W)^{m}\|_2=\|U(I-D)^{m}(I-D-I)U^Ts\|_2= \\
  \|(I-D)^{m}(I-D-I)U^Ts\|_2 &=& \|(I-D)^{m}(I-D-I)\widetilde{s}\|_2
\end{eqnarray}
since $U$ is a unitary matrix and where $\widetilde{s}=U^T s$.

Given a permutation matrix $P$ such that the entries of the diagonal $PDP^T$ are the eigenvalues of $W$ in decreasing order of magnitude, starting from $\lambda_0=1$, and assuming that $W$ has $k$ zero eigenvalues, where $k$ is a number in the set $\in\{0,\ 1,\ldots,\ n-1\}$, then
\begin{eqnarray}
    \nonumber \|(I-D)^{m}(I-D-I)\widetilde{s}\|_2 & \leq &  \left\|P\left[
                                          \begin{array}{ccccccc}
                                             0 &   &   &   &   &   &   \\
                                              &  (1-\lambda_1 )^m \lambda_1 &   &   &   &   &   \\
                                              &   & \ddots  &   &   &   &   \\
                                              &   &   &  (1-\lambda_{n-1-k} )^m \lambda_{n-1-k} &   &   &   \\
                                              &   &   &   &  0 &   &   \\
                                              &   &   &   &   &  \ddots &   \\
                                              &   &   &   &   &   & 0  \\
                                          \end{array}
                                        \right]P^T\left[
                                                 \begin{array}{c}
                                                   \|\widetilde{s}\|_\infty \\
                                                   \vdots \\
                                                   \|\widetilde{s}\|_\infty \\
                                                 \end{array}
                                               \right]
                                        \right\|_2 \\
  &\leq & {\sqrt{n-1-k}} \left(1-\frac{1}{m+1} \right)^m \frac{1}{m+1}  \|\widetilde{s}\|_\infty = {\sqrt{n-1-k}} \frac{m^m}{(m+1)^{m+1}}  \|\widetilde{s}\|_\infty
\end{eqnarray}
because the function $(1-\lambda)^m \lambda$ achieves its maximum at $\lambda=\frac{1}{m+1}$ for $\lambda\in[0,\ 1]$.

Hence the stopping criterion \eqref{eq:Discrete_Abs_StopCond} is fulfilled for $N_0$ minimum natural number such that \ref{eq:N0_discrete} holds true.
\end{proof}

We observe that, as we mentioned earlier, since (\ref{eq:Lambdas}) is also the Discrete Fourier Transform (DFT) formula of the sequence $\{c_{1q}\}_{q=0,\ldots,n-1}$, it follows that the eigenvalues of $W=\left[w_{pq}\right]_{p,\ q=0,\ldots,n-1}$,  can be computed directly as the DFT of the sequence $\{w_{1q}\}_{q=0,\ldots,n-1}$, by means of the Fast Fourier Transform (FFT). This calculation can be done ``off line'', in fact, once the filter shape $w$ has been fixed, we can compute and store its FFT for different values of the size of its support. This fact, together with other previous results, can be used to improve the efficiency of the method as explained in the following section.

It is interesting to notice that each IMF is generated as a linear combination of elements in an orthonormal basis. Therefore we can regard the IMFs as elements of a frame which allows to decompose a given signal into a few significant components. From this prospective the IF algorithm can be viewed as a method that automatically produces elements of a frame associated with a signal. The possible connections between IF and the frame theory are fascinating, but out of the scope of the present work. We plan to follow this direction of research in a future work.

Regarding the DIF outer loop convergence they hold true the same results described in Section \ref{subsec:outerLoopIF} for the continuous setting. In fact, while the inner loop of the IF algorithm requires a discretization to deal with discrete signals, the outer loop does not require any form of discretization and it works the very same as in the continuous setting.

\section{Efficient implementation of the DIF algorithm}\label{sec:speedup}

In this section we want to review some ideas for an efficient implementation of the DIF algorithm applied to the decomposition of a signal $s$ of length $n$.
We underline that the following ideas apply only for periodical extension of the signal at the boundaries.

We start from Theorem \ref{thm:DIF_conv_stop} which allows to compute each IMF as fast as the FFT of a signal of length $n$.
The first idea is to precompute the number of iterations needed to achieve the required accuracy $\delta$ in the computation of a certain IMF. This number of iterations can be approximated by the minimum $N_0\in\N$ satisfying the inequality \eqref{eq:N0_discrete}. Then we can compute the IMF using \eqref{eq:IMF1_direct} where the eigenvalues $\{\lambda_k\}_{k=1,\ 2,\ldots,\ n}$ can be evaluated using \eqref{eq:Lambdas}, or by means of the Fast Fourier Transform since \eqref{eq:Lambdas} is equivalent to the Discrete Fourier Transform of the sequence $\{w_{1q}\}_{q=0,\ldots,n-1}$. Furthermore we recall that $U^Ts$ is the DFT of $s$ that can be computed using the FFT algorithm, whose computational complexity is $n\log(n)$, and that multiplying on the left by the matrix $U$ is equivalent to computing the Inverse DFT (IDFT) which can be done using the inverse FFT. Hence the IMF can be computed in one step as
\begin{equation}\label{eq:One step_algo}
    \textrm{IMF} = \sum_{k=0}^{n-1} u_k (1-\lambda_k)^{N_0}\sigma_k = \textrm{IDFT}\left((I-D)^{N_0}\textrm{DFT}(s)\right)
\end{equation}
where $\sigma_k$ represents the $k$-th element of the DFT of the signal $s$.

The proposed a priori calculation of $N_0\in\N$ as the minimum value satisfying the inequality \eqref{eq:N0_discrete} is fast and easy, but provides only with an overestimation of the real number of iterations required. In order to compute the actual number of iterations required we can compute \eqref{eq:One step_algo} for subsequently bigger values of $N_0\in\N$ and stop whenever the quantity $SD$ defined in \eqref{eq:SD} is less or equal to $\delta$.
This is done in the so called Fast Iterative Filtering (FIF) method implemented for Matlab and available online\footnote{\url{www.cicone.com}}. By exploiting the FFT we speed up the calculations significantly. For a vector of tenths of millions of points the computational time passes from roughly two days for the standard IF to less than an hour on a personal computer with the FIF algorithm.

We point out that we can also precompute the eigenvalues $\lambda_k$ corresponding to any possible scaling of a filter $w$. In doing so we can reduce even further the computational time of the algorithm.

\section{Conclusions and Outlook}\label{sec:Conclusions}

In this work we tackle the problem of a complete analysis of the IF algorithm both in the continuous and discrete setting. In particular in the continuous setting we show how IF can decompose a signal into a finite number of so called IMFs and that each IMF contains frequencies of the original signal filtered in a ``smart way''.

In the discrete setting we prove that the DIF method is also convergente and, in the case of periodical extension at the boundaries of the given signal, we provide an explicit formula for the a priori calculation of each IMF. From this equation it follows that each IMF is a smart summation of eigenvectors of a circulant matrix.

We show that no fake oscillations can be produced neither in the continuous nor in the discrete setting.

From the properties of the DIF algorithm and the explicit formula for the IMFs produced by this method and derived in this work, we propose new ideas that has been directly incorporated in the implemented algorithm in order to increase its efficiency and reduce its computational complexity. The result is the so called FIF method\footnote{\url{www.cicone.com}} which allows to quickly decompose a signal by means of the FFT. This is an important result in this area of research which opens the doors to an almost instantaneous analysis of non stationary signals.

There are several open problems that remain unsolved. First of all from the proposed analysis it is clear that different filter functions have different Fourier transform and hence the decomposition produced by IF and DIF algorithms is directly influenced by this choice. In a future work we plan to study more in details the connections between the shape of the filters and the quality of the decomposition produced by these methods.

In the current work we analyzed the DIF assuming a periodical extension of the signals at the boundaries. We plan to study in a future work the behavior of the DIF method in the case of reflective, antireflective and other boundaries extensions of a signal.

Based on the numerical evidence \cite{cicone2016adaptive,cicone2017multidimensional} we claim that the Iterative Filtering method is stable under perturbations of the signal. We plan to study rigorously such stability in a future work.

The results about the DIF algorithm convergence suggest that the method allows, in general, to automatically generate a frame associated with a given signal. We plan to further analyze this connection in a future work.

Finally we recall that it is still an open problem how to extend all the results obtained for the Iterative Filtering technique to the case of the Adaptive Local Iterative Filtering method, whose convergence and stability analysis is still under investigation \cite{cicone2017spectral,cicone2016adaptive,cicone2017Geophysics}.


\section*{Acknowledgments}
This work was supported by NSF Awards DMS--1620345, DMS--1830225, ONR Award N00014--18--1--2852, the Istituto Nazionale di Alta Matematica (INdAM) ``INdAM Fellowships in Mathematics and/or Applications cofunded by Marie Curie Actions'', FP7--PEOPLE--2012--COFUND, Grant agreement n. PCOFUND--GA--2012--600198.


\begin{thebibliography}{10}

\bibitem{cicone2017dummies}
A.~Cicone.
\newblock Nonstationary signal decomposition for dummies.
\newblock In {\em Advances in Mathematical Methods and High Performance
  Computing}. in print.

\bibitem{cicone2017BC}
A.~Cicone and P.~Dell'Acqua.
\newblock Study of boundary conditions in the iterative filtering method for
  the decomposition of nonstationary signals.
\newblock {\em preprint}, 2018.

\bibitem{cicone2017spectral}
A.~Cicone, C.~Garoni, and Serra-Capizzano S.
\newblock Spectral and convergence analysis of the discrete alif method.
\newblock {\em preprint}, 2018.

\bibitem{cicone2016adaptive}
A.~Cicone, J.~Liu, and H.~Zhou.
\newblock Adaptive local iterative filtering for signal decomposition and
  instantaneous frequency analysis.
\newblock {\em Appl. Comput. Harmon. Anal.}, 41(2):384--411, 2016.

\bibitem{cicone2016hyperspectral}
A.~Cicone, J.~Liu, and H.~Zhou.
\newblock Hyperspectral chemical plume detection algorithms based on
  multidimensional iterative filtering decomposition.
\newblock {\em Phil. Trans. R. Soc. A: Math. Phys. Eng. Sci.},
  374(2065):2015.0196, 2016.

\bibitem{cicone2017multidimensional}
A.~Cicone and H.~Zhou.
\newblock Multidimensional iterative filtering method for the decomposition of
  high-dimensional non-stationary signals.
\newblock {\em Numer. Math. Theory Methods Appl.}, 10(2):278--298, 2017.

\bibitem{huang2009convergence}
C.~Huang, L.~Yang, and Y.~Wang.
\newblock Convergence of a convolution-filtering-based algorithm for empirical
  mode decomposition.
\newblock {\em Advances in Adaptive Data Analysis}, 1(04):561--571, 2009.

\bibitem{huang1998empirical}
N.E. Huang, Z.~Shen, S.R. Long, M.C. Wu, H.H. Shih, Q.~Zheng, N.C. Yen, C.C.
  Tung, and H.H. Liu.
\newblock The empirical mode decomposition and the hilbert spectrum for
  nonlinear and non-stationary time series analysis.
\newblock {\em Proceedings of the Royal Society of London. Series A:
  Mathematical, Physical and Engineering Sciences}, 454(1971):903, 1998.

\bibitem{lin2009iterative}
L.~Lin, Y.~Wang, and H.~Zhou.
\newblock Iterative filtering as an alternative algorithm for empirical mode
  decomposition.
\newblock {\em Advances in Adaptive Data Analysis}, 1(4):543--560, 2009.

\bibitem{cicone2017Geophysics}
M.~Piersanti, M.~Materassi, A.~Cicone, L.~Spogli, H.~Zhou, and R.~G. Ezquer.
\newblock Adaptive local iterative filtering: a promising technique for the
  analysis of non-stationary signals.
\newblock {\em Journal of Geophysical Research -- Space Physics},
  123(1):1031--1046, 2018.

\bibitem{sfarra2019thermal}
S.~Sfarra, A.~Cicone, B.~Yousefi, C.~Ibarra-Castanedo, S.~Perillia, and
  X.~Maldaguef.
\newblock Improving the detection of thermal bridges in buildings via on-site
  infrared thermography: the potentialities of innovative mathematical tools.
\newblock {\em Energy and Buildings}, in print.

\bibitem{wang2012iterative}
Y.~Wang, G.-W. Wei, and S.~Yang.
\newblock Iterative filtering decomposition based on local spectral evolution
  kernel.
\newblock {\em Journal of scientific computing}, 50(3):629--664, 2012.

\bibitem{wang2013convergence}
Y.~Wang and Z.~Zhou.
\newblock On the convergence of iterative filtering empirical mode
  decomposition.
\newblock In {\em Excursions in Harmonic Analysis, Volume 2}, pages 157--172.
  Springer, 2013.

\end{thebibliography}
\end{document}